\documentclass{article}%
\usepackage{amsfonts}
\usepackage{amsmath}
\usepackage{colortbl}
\usepackage[letterpaper]{geometry}
\usepackage{amssymb}
\usepackage{graphicx}%
\newtheorem{theorem}{Theorem}

\newtheorem{corollary}[theorem]{Corollary}

\newtheorem{example}[theorem]{Example}

\newenvironment{proof}[1][Proof]{\noindent\textbf{#1.} }{}
\begin{document}

\title{Spread polynomials, rotations and the butterfly effect}
\date{}
\author{Shuxiang Goh\\35 Firmin Court\\Mermaid Waters \\Queensland 4218\\Australia
\and N. J. Wildberger\\School of Mathematics\\and Statistics, UNSW \\Sydney 2052 \\Australia}
\maketitle

\begin{abstract}
The spread between two lines in rational trigonometry replaces the concept of
angle, allowing the complete specification of many geometrical and dynamical
situations which have traditionally been viewed approximately. This paper
investigates the case of powers of a rational spread rotation, and in
particular, a curious periodicity in the prime power decomposition of the
associated values of the spread polynomials, which are the analogs in rational
trigonometry of the Chebyshev polynomials of the first kind. Rational
trigonometry over finite fields plays a role, together with non-Euclidean geometries.

\end{abstract}

\section{Introduction}

This paper investigates the role of the \textit{spread polynomials}
$S_{n}\left(  s\right)  $ of \textit{rational trigonometry} in understanding a
simple dynamical system arising in elementary geometry. Spread polynomials
were introduced in \cite{Wild}, and arise naturally when Euclidean geometry is
studied \textit{algebraically,} with angles replaced by spreads, and they play
also a role in geometry over finite fields, and in fact over general fields
not of characteristic two. Furthermore their importance extends to
\textit{elliptic }and \textit{hyperbolic }geometries, see \cite{WildHyp}.
Spread polynomials also have interesting number theoretical properties that
set them apart from Chebyshev and other orthogonal polynomials, see
\cite{Wild} and \cite{Goh}.

In this paper we study a phenomenon which occurs when we analyze iterates, or
powers, of a particular rotation $\rho$ in the Euclidean plane. If $\rho$ is
the rotation by an angle $\theta$ which is a rational multiple of $\pi,$ then
the subsequent powers $\rho^{n}$ are easily understood. However when $\theta$
is an irrational multiple of $\pi$, the situation exhibits chaotic features,
and illustrates the `butterfly effect': the future state of the system is
determined by the accuracy by which we know the real numbers which specify
initial conditions, and even a small error will eventually result in complete
uncertainty---this is the standard view.

With rational trigonometry, \textit{increasing uncertainty} is replaced with
\textit{increasing complexity}. By measuring rotations $\rho$ with a
\textit{spread }$s$, not an \textit{angle} $\theta$, the spread polynomials
arise naturally from iterates, and the chaos disappears, replaced instead by
an increasing escalation in the size and complexity of the numbers
$S_{n}\left(  s\right)  $ that describe the evolution of the system. These
numbers are far from random, and they tend to incorporate interesting, and
indeed mysterious, number theoretic information.

For a rational spread $s,$ we will prove a theorem about the spreads
$S_{n}\left(  s\right)  $ associated to the powers $\rho^{n}$ of the rotation
$\rho$ with initial spread $s$. Such rational spreads are common in geometry,
for example any configuration of lines with rational equations has only
rational spreads. Of particular interest is that we will need to examine
\textit{finite geometries} in order to understand the situation over the
rational numbers, and that forms of \textit{non-Euclidean geometry} also
arise. Our work can also be viewed in the context of investigations of
Chebyshev polynomials over a finite field, initiated already in \cite{Schur};
see also \cite{Bang} and \cite{Rankin}.

\subsection{Powers of a rotation}

To provide some motivation for the spread polynomials, recall that with the
ISO size standards, an $A4$ sheet of paper has proportions in the ratio
$\sqrt{2}:1.$ So if you fold a piece in two lengthwise, the result has the
same proportions but only one half of the area, and is called $A5$ size. The
angle in degrees between the long side $l_{0}$ and the diagonal $l_{1}$ of an
$A4$ sheet may be described by the formula
\[
\theta_{ISO}\equiv\left(  180\arcsin\sqrt{1/3}\right)  /\pi.
\]
This has a numerical value in degrees of approximately $\theta_{ISO}%
\approx35.26.$

Suppose we rotate $l_{1}$ by $\theta_{ISO}$ to get $l_{2},$ then rotate
$l_{2}$ by $\theta_{ISO}$ to get $l_{3},$ and so on. We could also say that
$l_{2}$ is the reflection of $l_{0}$ in $l_{1},$ that $l_{3}$ is the
reflection of $l_{1}$ in $l_{2},$ and so on. This gives a sequence\textbf{\ }%
$l_{1},l_{2},\cdots$ of concurrent lines whose respective angles with $l_{0}$
form the sequence $\theta_{ISO},2\theta_{ISO},3\theta_{ISO},\cdots.$ Since
$\theta_{ISO}$ is not a rational multiple of $\pi,$ these sequences of lines
and angles will never repeat, and so the lines $l_{n}$ are all distinct, as
are the angles $n\theta_{ISO}$.%
\begin{figure}
[h]
\begin{center}
\includegraphics[
height=6.8442cm,
width=6.9852cm
]%
{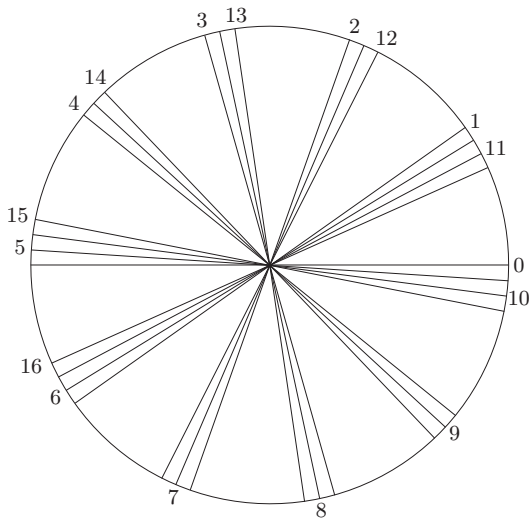}%
\caption{Multiples of $\theta_{ISO}$, or equivalently $s=1/3$}%
\label{Multiples}%
\end{center}
\end{figure}

Figure \ref{Multiples} shows the first sixteen such lines, exhibiting the
famous \textit{three gaps phenomenon, }which was originally a conjecture of
Steinhaus, (see for example \cite{Rav}): for any $n,$ there are at most three
different angles found between any two adjacent lines in the set $\left\{
l_{0},l_{1},l_{2},\cdots l_{n}\right\}  $.

The angle between $l_{0}$ and $l_{16}$ is approximately
\[
16\theta_{ISO}\approx16\times35.26\operatorname{mod}360=\allowbreak
564.\,\allowbreak16-360=\allowbreak204.\,\allowbreak16.
\]
We have lost information, however, and with the better initial approximation
$\theta_{ISO}\approx35.\,\allowbreak264\,389\,682$ we get instead
\[
16\theta_{ISO}\approx204.\,\allowbreak230\,235.
\]
But no matter what accuracy we have for $\theta_{ISO},$ computing larger and
larger multiples of it, mod $360,$ inevitably results in increasing
error---until after a finite number of multiples \textit{all knowledge of the
position is lost}. This is a simple example of the
well-known\textit{\ butterfly effect}, a feature of a wide range of dynamical
systems, and an inevitable consequence of working in the framework of
classical trigonometry, which generally deals only with \textit{approximations
}to real values. The usual idea is that our understanding of the future
evolution of many completely specified systems is limited by the precision
with which we know the real numbers that appear in the initial conditions.

With \textit{rational trigonometry, }introduced in\textit{\ }\cite{Wild}%
\textit{, }we set our sights higher---we aim to describe such a
system\textit{\ completely precisely, }until we run out of computing power,
memory space, or patience. The price we pay for \textit{more accuracy}
is---\textit{more complexity}. For dynamical systems which can be expressed
using rational numbers and polynomial transformations, the future evolution is
completely known, but will generally be increasingly difficult and expensive
to write down as time goes on. We will show that rotation by the angle
$\theta_{ISO}$ exhibits this phenomenon, and brings out interesting number
theoretical questions.

It will be important for us to consider not just rotations in the Euclidean
plane, but also over finite prime fields. For this it is often more convenient
to work not with a unit circle, but rather with the \textit{associated
projective line}. The metrical geometry of the projective line is described by
a \textit{projective version of rational trigonometry}, which also extends to
higher projective spaces, and it is in terms of this that we formulate
rotations and reflections.

\subsection{Basic rational trigonometry}

Here is a quick review of some main ideas from affine rational trigonometry in
the plane; the main reference is \cite{Wild}, see also \cite{WildHorizons}. We
work over a general field, not of characteristic two. The primary measurement
is the \textit{quadrance }between two points $A_{1}\equiv\left[  x_{1}%
,y_{1}\right]  $ and $A_{2}\equiv\left[  x_{2},y_{2}\right]  $, defined by%
\[
Q\left(  A_{1},A_{2}\right)  \equiv\left(  x_{2}-x_{1}\right)  ^{2}+\left(
y_{2}-y_{1}\right)  ^{2}.
\]
Over the real numbers, quadrance is the square of distance---or better yet
\textit{distance is the square root of quadrance}. However in rational
trigonometry we wish to avoid square roots, and so distance plays no role.
This allows the theory to extend to more general fields.

The short side, long side and diagonal of an $A4$ sheet of paper form a right
triangle with quadrances $Q_{1},Q_{2}$ and $Q_{3}$ in the ratio $1:2:3, $ and
Pythagoras' theorem takes the form
\[
Q_{1}+Q_{2}=Q_{3}.
\]

In rational trigonometry the separation of two lines is measured by a
\textit{spread}, not an \textit{angle}. A line $l$ with equation $ax+by+c=0$
is a \textbf{null line} precisely when $a^{2}+b^{2}=0.$ The \textbf{spread}
between two non-null lines with equations
\[%
\begin{tabular}
[c]{lllll}%
$a_{1}x+b_{1}y+c_{1}=0$ &  & \textrm{and } &  & $a_{2}x+b_{2}y+c_{2}=0$%
\end{tabular}
\]
is defined to be the number%
\[
s\equiv\frac{\left(  a_{1}b_{2}-a_{2}b_{1}\right)  ^{2}}{\left(  a_{1}%
^{2}+b_{1}^{2}\right)  \left(  a_{2}^{2}+b_{2}^{2}\right)  }.
\]
This formula also gives the spread between lines with direction vectors
$\left(  a_{1},b_{1}\right)  $ and $\left(  a_{2},b_{2}\right)  $. Over the
real numbers the spread $s$ between two lines lies between $0$ and $1,$ being
$0$ when lines are parallel, $1$ when lines are perpendicular, and in general
the square of the sine of an angle $\theta$ between them. There are many such
possible angles (e.g. $\theta,$ $\pi-\theta,2\pi+\theta,3\pi-\theta,\cdots$),
but the square of the sine---the \textit{spread}---is the same for all.

The spread between two non-null lines can be expressed as a ratio of two
quadrances---an opposite quadrance to a hypotenuse quadrance in a right
triangle formed from those lines, so for example the spread between the long
side $l_{0}$ of a piece of $A4$ paper and its diagonal $l_{1}$ is $s\left(
l_{0},l_{1}\right)  =s=1/3.$ Unlike the notion of angle, spread really is
defined between \textit{lines}, not \textit{rays}. Using spreads instead of
angles makes much of the study of triangles---that is, \textit{trigonometry}%
---dramatically simpler. This is explained at length in \cite{Wild}.

Define a number $s$ to be a \textbf{spread number }precisely when $s\left(
1-s\right)  $ is a square in the field. The Spread number theorem
(\cite[Chapter 6]{Wild}) asserts that a number $s$ is the spread between two
lines precisely when $s$ is a spread number. In the finite prime field
$\mathbb{F}_{p}$ of $p$ elements, there are $\left(  p+3\right)  /2$ or
$\left(  p+1\right)  /2$ spread numbers, depending respectively on whether $p$
is congruent to $3$ $\operatorname{mod}4$ or $1$ $\operatorname{mod}4.$

\subsection{Definition of the spread polynomials}

The \textbf{spread polynomials }$S_{n}\left(  s\right)  $ arise when we
consider the rational analog of \textit{multiples of an angle}. They may be
defined recursively, over a general field, without any reference to geometry
as follows:
\begin{align}
S_{0}\left(  s\right)   &  \equiv0\nonumber\\
S_{1}\left(  s\right)   &  \equiv s\nonumber\\
S_{n}\left(  s\right)   &  \equiv2\left(  1-2s\right)  S_{n-1}\left(
s\right)  -S_{n-2}\left(  s\right)  +2s. \label{Recursive}%
\end{align}

The coefficient of $s^{n}$ in $S_{n}\left(  s\right)  $ is a power of four, so
over any field not of characteristic two, $S_{n}\left(  s\right)  $ is a
polynomial of degree $n$. The generating function for the spread polynomials
was computed by M. Hirschhorn, it is
\begin{equation}
S\left(  s,t\right)  \equiv\sum_{n=0}^{\infty}S_{n}\left(  s\right)
t^{n}=\frac{ts\left(  1+t\right)  }{\left(  1-t\right)  \left(  1-2t+t^{2}%
+4ts\right)  }. \label{SGenerating}%
\end{equation}

The spread polynomials are intimately linked to geometry, in the following
sense. If two intersecting lines $l_{0}$ and $l_{1}$ in the plane make a
spread of $s,$ then we may reflect $l_{0}$ in $l_{1}$ to obtain $l_{2},$
reflect $l_{1}$ in $l_{2}$ to obtain $l_{3},$ and so on. We will see that the
spread $s\left(  l_{0},l_{n}\right)  $ is then $S_{n}\left(  s\right)  $. For
this description we only need the lines through the origin, so it is really a
statement about the associated \textit{projective line}. In fact our statement
of this fact, the Spread of a power theorem, will be stated in sufficient
generality to include also some non-Euclidean geometries, and will involve a
multiplicative structure on the \textit{non-null points of the projective
line}.

Another key fact about the spread polynomials, which connects with their
geometric interpretation, is the following.

\medskip

\textbf{Theorem (Spread composition)} \textit{For any natural numbers }%
$n$\textit{\ and }$m,$\textit{\ }%
\[
S_{n}\circ S_{m}=S_{nm}.
\]

We will give a proof later using the relations between spreads and rotations;
for another see \cite[page 110]{Wild}.

\subsection{Table and graphs of spread polynomials}

Here are the first few spread polynomials. Note that $S_{2}\left(  s\right)  $
is the \textit{logistic map}.%
\begin{align*}
S_{0}\left(  s\right)   &  =0\\
S_{1}\left(  s\right)   &  =s\\
S_{2}\left(  s\right)   &  =4s-4s^{2}=4s\left(  1-s\right) \\
S_{3}\left(  s\right)   &  =9s-24s^{2}+16s^{3}=s\left(  3-4s\right)  ^{2}\\
S_{4}\left(  s\right)   &  =16s-80s^{2}+128s^{3}-64s^{4}=16s\left(
1-s\right)  \left(  1-2s\right)  ^{2}\\
S_{5}\left(  s\right)   &  =25s-200s^{2}+560s^{3}-640s^{4}+256s^{5}=s\left(
5-20s+16s^{2}\right)  ^{2}\\
S_{6}\left(  s\right)   &  =4s\left(  1-s\right)  \left(  3-4s\right)
^{2}\left(  1-4s\right)  ^{2}\\
S_{7}\left(  s\right)   &  =s\left(  7-56s+112s^{2}-64s^{3}\right)  ^{2}\\
S_{8}\left(  s\right)   &  =\allowbreak64s\left(  1-s\right)  \left(
1-2s\right)  ^{2}\left(  1-8s+8s^{2}\right)  ^{2}\allowbreak\\
S_{9}\left(  s\right)   &  =s\left(  3-4s\right)  ^{2}\left(  3-36s+96s^{2}%
-64s^{3}\right)  ^{2}\\
S_{10}\left(  s\right)   &  =4s\left(  1-s\right)  \left(  1-12s+16s^{2}%
\right)  ^{2}\left(  5-20s+16s^{2}\right)  ^{2}\allowbreak
\end{align*}

A remarkable fact, already suggested by this list, is that the spread
polynomials factor in an interesting way, indeed in a more pleasant fashion
than the Chebyshev polynomials. We will establish this in a future paper.

The next Figures show the first eight and twenty five spread polynomials over
the real numbers in the range $0\leq s\leq1.$ Note the interesting ghost
patterns that begin to appear as we increase the number of polynomials shown;
these are related to Lissajous curves, and such a phenomenon occurs also for
Chebyshev polynomials.

Observe also that the spread polynomials are positive in this range, so do not
form an orthogonal family of polynomials in the usual sense unless they are
translated vertically.
\begin{center}
\includegraphics[
height=4.6406cm,
width=7.2831cm
]%
{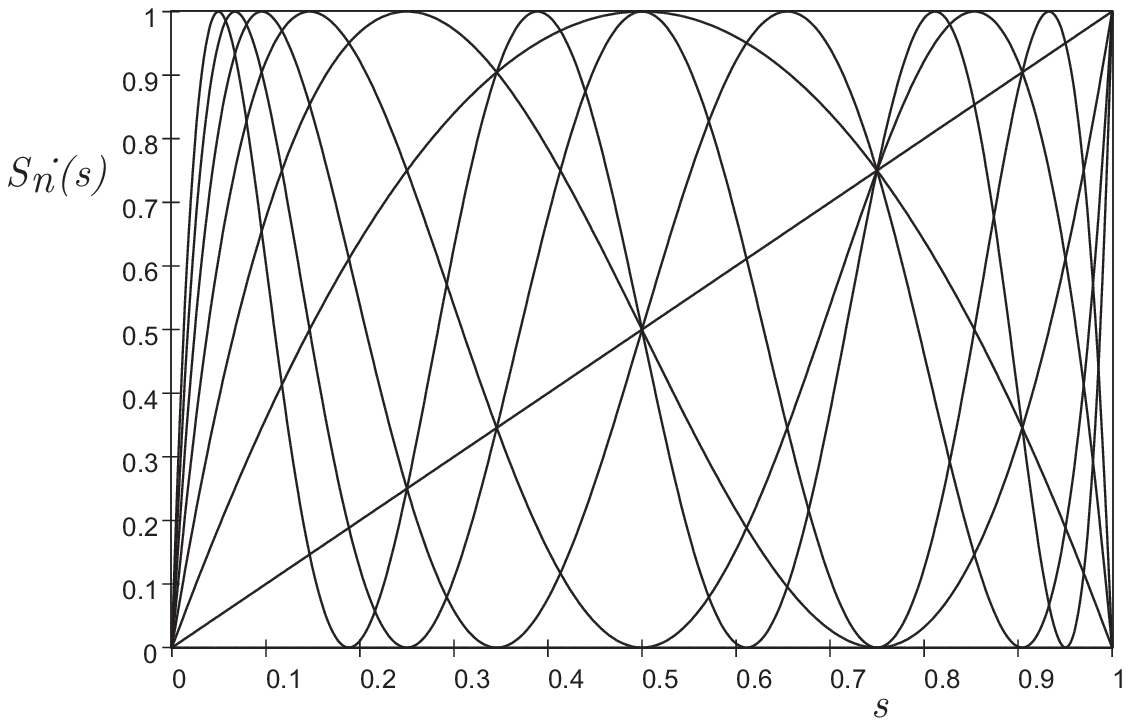}%
\\
Figure 2: The first eight spread polynomials $S_{0}\left(  s\right)  $ to
$S_{7}\left(  s\right)  $%
\end{center}
\qquad\qquad%

\begin{center}
\includegraphics[
height=4.5603cm,
width=7.1377cm
]%
{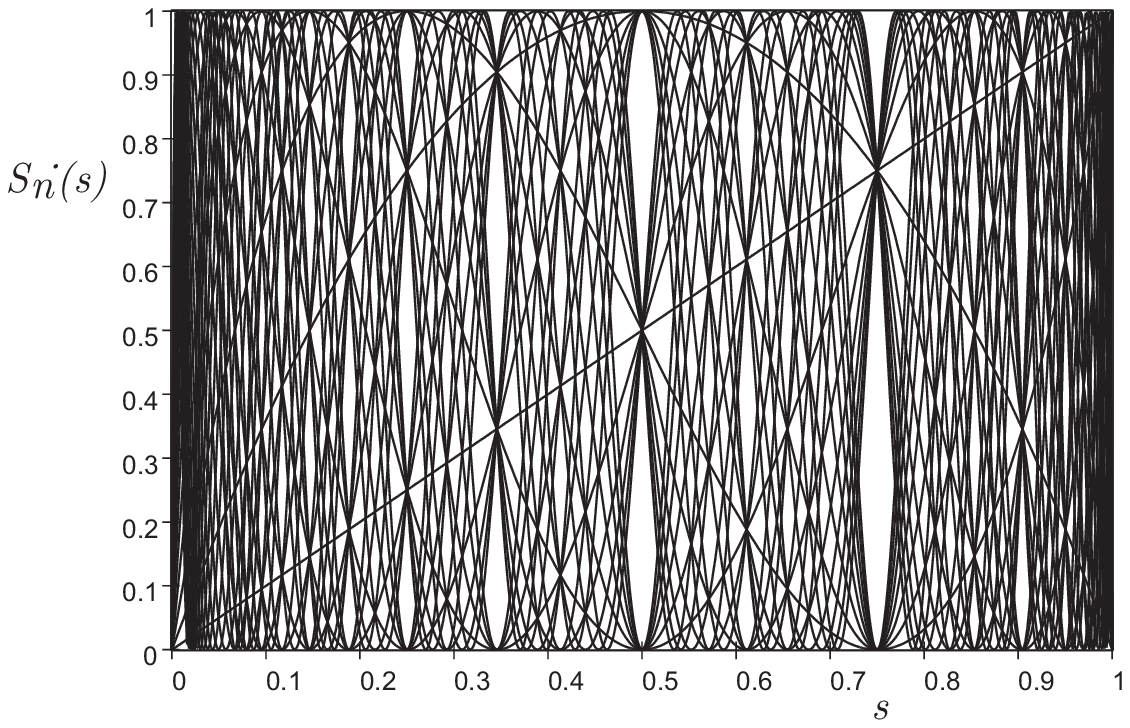}%
\\
Figure 3: The first twenty five spread polynomials
\end{center}

\begin{example}
Here are some values of spread polynomials in the field $\mathbb{F}_{5}.$ The
spread numbers in this field are $0,1$ and $3.$%
\[%
\begin{tabular}
[c]{l||lllll}%
$s$ & $0$ & $1$ & $2$ & $3$ & $4$\\\hline\hline
$S_{0}\left(  s\right)  $ & $0$ & $0$ & $0$ & $0$ & $0$\\
$S_{1}\left(  s\right)  $ & $0$ & $1$ & $2$ & $3$ & $4$\\
$S_{2}\left(  s\right)  $ & $0$ & $0$ & $2$ & $1$ & $2$\\
$S_{3}\left(  s\right)  $ & $0$ & $1$ & $0$ & $3$ & $1$\\
$S_{4}\left(  s\right)  $ & $0$ & $0$ & $2$ & $0$ & $2$\\
$S_{5}\left(  s\right)  $ & $0$ & $1$ & $2$ & $3$ & $4$%
\end{tabular}
\qquad\qquad%
\begin{tabular}
[c]{l||lllll}%
$s$ & $0$ & $1$ & $2$ & $3$ & $4$\\\hline\hline
$S_{6}\left(  s\right)  $ & $0$ & $0$ & $0$ & $1$ & $0$\\
$S_{7}\left(  s\right)  $ & $0$ & $1$ & $2$ & $3$ & $4$\\
$S_{8}\left(  s\right)  $ & $0$ & $0$ & $2$ & $0$ & $2$\\
$S_{9}\left(  s\right)  $ & $0$ & $1$ & $0$ & $3$ & $1$\\
$S_{10}\left(  s\right)  $ & $0$ & $0$ & $2$ & $1$ & $2$\\
$S_{11}\left(  s\right)  $ & $0$ & $1$ & $2$ & $3$ & $4$%
\end{tabular}
\]
The pattern repeats with period $12$, that is $S_{n}\left(  s\right)
=S_{n+12}\left(  s\right)  $ for all $n$ and for all $s.$ However for values
of $s$ which are spread numbers, $S_{n}\left(  s\right)  =S_{n+4}\left(
s\right)  $ for all $n,$ while for values of $s$ which are non-spread numbers
$S_{n}\left(  s\right)  =S_{n+6}\left(  s\right)  $ for all $n.%
\hspace{.1in}\diamond
$
\end{example}

\begin{example}
Here are the values of the spread polynomials in the field $\mathbb{F}_{7}.$
The spread numbers in this field are $0,1,3,4$ and $5.$%
\[
\medskip%
\begin{tabular}
[c]{l||lllllll}%
$s$ & $0$ & $1$ & $2$ & $3$ & $4$ & $5$ & $6$\\\hline\hline
$S_{0}\left(  s\right)  $ & $0$ & $0$ & $0$ & $0$ & $0$ & $0$ & $0$\\
$S_{1}\left(  s\right)  $ & $0$ & $1$ & $2$ & $3$ & $4$ & $5$ & $6$\\
$S_{2}\left(  s\right)  $ & $0$ & $0$ & $6$ & $4$ & $1$ & $4$ & $6$\\
$S_{3}\left(  s\right)  $ & $0$ & $1$ & $1$ & $5$ & $4$ & $3$ & $0$\\
$S_{4}\left(  s\right)  $ & $0$ & $0$ & $6$ & $1$ & $0$ & $1$ & $6$\\
$S_{5}\left(  s\right)  $ & $0$ & $1$ & $2$ & $5$ & $4$ & $3$ & $6$\\
$S_{6}\left(  s\right)  $ & $0$ & $0$ & $0$ & $4$ & $1$ & $4$ & $0$\\
$S_{7}\left(  s\right)  $ & $0$ & $1$ & $2$ & $3$ & $4$ & $5$ & $6$\\
$S_{8}\left(  s\right)  $ & $0$ & $0$ & $6$ & $0$ & $0$ & $0$ & $6$\\
$S_{9}\left(  s\right)  $ & $0$ & $1$ & $1$ & $3$ & $4$ & $5$ & $0$\\
$S_{10}\left(  s\right)  $ & $0$ & $0$ & $6$ & $4$ & $1$ & $4$ & $6$\\
$S_{11}\left(  s\right)  $ & $0$ & $1$ & $2$ & $5$ & $4$ & $3$ & $6$\\
$S_{12}\left(  s\right)  $ & $0$ & $0$ & $0$ & $1$ & $0$ & $1$ & $0$%
\end{tabular}
\qquad\qquad%
\begin{tabular}
[c]{l||lllllll}%
$s$ & $0$ & $1$ & $2$ & $3$ & $4$ & $5$ & $6$\\\hline\hline
$S_{13}\left(  s\right)  $ & $0$ & $1$ & $2$ & $5$ & $4$ & $3$ & $6$\\
$S_{14}\left(  s\right)  $ & $0$ & $0$ & $6$ & $4$ & $1$ & $4$ & $6$\\
$S_{15}\left(  s\right)  $ & $0$ & $1$ & $1$ & $3$ & $4$ & $5$ & $0$\\
$S_{16}\left(  s\right)  $ & $0$ & $0$ & $6$ & $0$ & $0$ & $0$ & $6$\\
$S_{17}\left(  s\right)  $ & $0$ & $1$ & $2$ & $3$ & $4$ & $5$ & $6$\\
$S_{18}\left(  s\right)  $ & $0$ & $0$ & $0$ & $4$ & $1$ & $4$ & $0$\\
$S_{19}\left(  s\right)  $ & $0$ & $1$ & $2$ & $5$ & $4$ & $3$ & $6$\\
$S_{20}\left(  s\right)  $ & $0$ & $0$ & $6$ & $1$ & $0$ & $1$ & $6$\\
$S_{21}\left(  s\right)  $ & $0$ & $1$ & $1$ & $5$ & $4$ & $3$ & $0$\\
$S_{22}\left(  s\right)  $ & $0$ & $0$ & $6$ & $4$ & $1$ & $4$ & $6$\\
$S_{23}\left(  s\right)  $ & $0$ & $1$ & $2$ & $3$ & $4$ & $5$ & $6$\\
$S_{24}\left(  s\right)  $ & $0$ & $0$ & $0$ & $0$ & $0$ & $0$ & $0$\\
$S_{25}\left(  s\right)  $ & $0$ & $1$ & $2$ & $3$ & $4$ & $5$ & $6$%
\end{tabular}
\]
\newline The pattern repeats with period $24,$ that is $S_{n}\left(  s\right)
=S_{n+24}\left(  s\right)  $ for all $n$ and for all $s.$ For values of $s$
which are spread numbers, however, $S_{n}\left(  s\right)  =S_{n+8}\left(
s\right)  $ for all $n,$ while for the non-spread numbers $S_{n}\left(
s\right)  =S_{n+6}\left(  s\right)  .%
\hspace{.1in}\diamond
$
\end{example}

\subsection{Spread polynomials evaluated at $s=1/3$}

Consider the spread $s=1/3$ formed by the long side $l_{0}$ and the diagonal
$l_{1}$ of an $A4$ sheet of paper, as discussed previously. The numbers
$S_{n}\left(  1/3\right)  =s\left(  l_{0},l_{n}\right)  $ are the spreads
formed by successive reflections, or rotations, of these two lines.

The recursive formula (\ref{Recursive}) allows us to calculate the following
list of prime power factorizations:%
\[%
\begin{tabular}
[c]{ll}%
$S_{1}\left(  1/3\right)  =3^{-1}$ & $S_{2}\left(  1/3\right)  =\allowbreak
2^{3}3^{-2}$\\
$S_{3}\left(  1/3\right)  =\allowbreak5^{2}3^{-3}$ & $S_{4}\left(  1/3\right)
=\allowbreak\allowbreak2^{5}3^{-4}$\\
$S_{5}\left(  1/3\right)  =\allowbreak3^{-5}$ & $S_{6}\left(  1/3\right)
=\allowbreak\allowbreak2^{3}5^{2}3^{-6}$\\
$S_{7}\left(  1/3\right)  =\allowbreak\allowbreak43^{2}3^{-7}$ & $S_{8}\left(
1/3\right)  =\allowbreak\allowbreak2^{7}7^{2}3^{-8}$\\
$S_{9}\left(  1/3\right)  =\allowbreak\allowbreak5^{2}19^{2}3^{-9}$ &
$S_{10}\left(  1/3\right)  =\allowbreak\allowbreak2^{3}11^{2}3^{-10}$\\
$S_{11}\left(  1/3\right)  =\allowbreak\allowbreak197^{2}3^{-11}$ &
$S_{12}\left(  1/3\right)  =\allowbreak\allowbreak2^{5}5^{2}23^{2}3^{-12}$\\
$S_{13}\left(  1/3\right)  =\allowbreak\allowbreak1249^{2}3^{-13}$ &
$S_{14}\left(  1/3\right)  =\allowbreak\allowbreak2^{3}13^{2}43^{2}3^{-14}$\\
$S_{15}\left(  1/3\right)  =\allowbreak\allowbreak5^{4}29^{2}3^{-15}$ &
$S_{16}\left(  1/3\right)  =\allowbreak\allowbreak2^{9}7^{2}17^{2}3^{-16}$\\
$S_{17}\left(  1/3\right)  =\allowbreak\allowbreak9791^{2}3^{-17}$ &
$S_{18}\left(  1/3\right)  =\allowbreak\allowbreak2^{3}5^{2}19^{2}%
73^{2}3^{-18}$\\
$S_{19}\left(  1/3\right)  =\allowbreak\allowbreak26107^{2}3^{-19}$ &
$S_{20}\left(  1/3\right)  =\allowbreak\allowbreak2^{5}11^{2}241^{2}3^{-20}$\\
$S_{21}\left(  1/3\right)  =\allowbreak\allowbreak5^{2}43^{2}167^{2}3^{-21}$ &
$S_{22}\left(  1/3\right)  =\allowbreak\allowbreak2^{3}197^{2}263^{2}3^{-22}%
$\\
$S_{23}\left(  1/3\right)  =\allowbreak\allowbreak139^{2}2207^{2}3^{-23}$ &
$S_{24}\left(  1/3\right)  =\allowbreak\allowbreak2^{7}5^{2}7^{2}23^{2}%
47^{2}3^{-24}$\\
$S_{25}\left(  1/3\right)  =\allowbreak\allowbreak149^{2}1949^{2}3^{-25}$ &
$S_{26}\left(  1/3\right)  =\allowbreak\allowbreak2^{3}131^{2}1249^{2}3^{-26}%
$\\
$S_{27}\left(  1/3\right)  =\allowbreak\allowbreak5^{2}19^{2}53^{2}%
433^{2}3^{-27}$ & $S_{28}\left(  1/3\right)  =\allowbreak\allowbreak
2^{5}13^{2}43^{2}1511^{2}3^{-28}$\\
$S_{29}\left(  1/3\right)  =\allowbreak\allowbreak6973919^{2}3^{-29}$ &
$S_{30}\left(  1/3\right)  =\allowbreak\allowbreak2^{3}5^{4}11^{2}%
29^{2}239^{2}3^{-30}\allowbreak$%
\end{tabular}
\]

Let's make some empirical observations about the above table. For each $n,$
$S_{n}\left(  1/3\right)  $ is a fraction whose denominator is $3^{n}.$ The
numerator is divisible by $2$ precisely when $n$ is even, and the power of $2
$ appearing is odd. The other factor of the numerator is a square. Prime
factors of the numerator seem to occur periodically. For an odd prime
$p\neq3,$ define $m\left(  p\right)  $ to be the smallest natural number $m$
such that for all positive multiples $n$ of $m,$ $S_{n}\left(  1/3\right)  $
has a factor of $p,$ if such an $m$ exists. From the table, we may guess that
this number for small primes $p$ is:%
\[%
\begin{tabular}
[c]{llll}%
$m\left(  5\right)  =3$ & $m\left(  7\right)  =8$ & $m\left(  11\right)  =10$
& $m\left(  13\right)  =14$\\
$m\left(  17\right)  =16$ & $m\left(  19\right)  =9$ & $m\left(  23\right)
=12$ & $m\left(  29\right)  =15.$%
\end{tabular}
\]
However larger primes also appear in the table, for example perhaps%
\[
m\left(  6973919\right)  =29.
\]
The aim of this paper is to try to begin to explain these numbers, and to show
that the phenomenon is not dependent on the initial spread $s=1/3$.

We adopt the convention that a rational number $\alpha$ is \textbf{divisible
by a prime }$p$ precisely when $\alpha=p\beta$ with $\beta$ a rational number
which can be expressed as $\beta=c/d$ with $c$ and $d$ integers, and $d$ not
divisible by $p.$ In this case we say $p$ is a \textbf{factor} of $\alpha.$
The following is the main result of this paper.

\medskip

\textbf{Theorem (Spread Periodicity) }\textit{For any rational number
}$s\equiv a/b$\textit{\ and any prime }$p$\textit{\ not dividing }%
$b,$\textit{\ there is a natural number }$m$\textit{\ such that }$S_{n}\left(
s\right)  $\textit{\ is divisible by }$p$\textit{\ precisely when }%
$m$\textit{\ divides }$n.$\textit{\ This number }$m$\textit{\ is a divisor of
either }$p-1$\textit{\ or }$p+1.$

\medskip

\textbf{Corollary }\textit{For any rational number }$s\equiv a/b$\textit{, any
prime }$p$\textit{\ not dividing }$b$ \textit{occurs infinitely often}
\textit{as a factor of the numbers }$S_{n}\left(  s\right)  $ for\textit{\ }%
$n=1,2,3,\cdots.$

\medskip

In addition, if we find a prime $p$ appearing as a factor of $S_{k}\left(
s\right)  $, then we can be sure that $p$ will appear as a factor of any
spread $S_{nk}\left(  s\right)  $, for $n=1,2,3,\cdots.$ For example, from the
above observations it follows that $S_{58}\left(  1/3\right)  $ is divisible
by $6973919.$ Note however that we are not able to address the more difficult
problem of determining $m\left(  p\right)  $ for a given $p.$

To prove the Spread periodicity theorem, we will explore the metrical geometry
of the \textit{one-dimensional projective line,} and show that finite
geometries, both \textit{Euclidean and non-Euclidean}, play a role.

\section{Geometry of the projective line}

The metrical structure of one-dimensional geometry over a general field
$\mathbb{F},$ not of characteristic two, was investigated recently in
\cite{WildAffine}. There are two distinctly different contexts:
\textit{affine} and \textit{projective}, and in this paper it is the
projective setting that is of primary interest, because we are interested in
rotations and reflections of one-dimensional subspaces of a two-dimensional
space $\mathbb{F}^{2}$, which essentially takes place in the one-dimensional
projective line.

A vector in $\mathbb{F}^{2}$ will typically be denoted $U$ or $V.$ If $U$ is a
non-zero vector, then $u=\left[  U\right]  $ represents the corresponding
\textbf{projective point}, or \textbf{p-point} for short, namely the line $OU
$ through $U$ and the origin $O\equiv\left[  0,0\right]  $. If $U\equiv\left[
x,y\right]  $ then we write $u=\left[  U\right]  \equiv\left[  x:y\right]  $,
with the usual convention for proportions that $\left[  x_{1}:y_{1}\right]
=\left[  x_{2}:y_{2}\right]  $ precisely when $x_{1}y_{2}-x_{2}y_{1}=0.$ The
projective points constitute the\textbf{\ projective line} $\mathbb{P}^{1}$.

Now fix a symmetric bilinear form on $\mathbb{F}^{2},$%
\begin{equation}
\left[  x_{1},y_{1}\right]  \cdot\left[  x_{2},y_{2}\right]  \equiv
ax_{1}x_{2}+b\left(  x_{1}y_{2}+x_{2}y_{1}\right)  +cy_{1}y_{2}
\label{Bilinear}%
\end{equation}
which is \textbf{non-degenerate}, that is
\[
ac-b^{2}\neq0.
\]
The projective point $u\equiv\left[  U\right]  $ is then \textbf{null}
precisely when $U\cdot U=0.$ It is important to realize that such a bilinear
form on $\mathbb{F}^{2}$ can naturally provide a \textit{metrical structure on
the associated projective line} $\mathbb{P}^{1}$.

The \textbf{projective quadrance}, or \textbf{p-quadrance} for short, between
the non-null p-points $u_{1}\equiv\left[  U_{1}\right]  $ and $u_{2}%
\equiv\left[  U_{2}\right]  $ is the number
\[
q\left(  u_{1},u_{2}\right)  \equiv1-\frac{\left(  U_{1}\cdot U_{2}\right)
^{2}}{\left(  U_{1}\cdot U_{1}\right)  \left(  U_{2}\cdot U_{2}\right)  }.
\]
This is well-defined, and if $u_{1}\equiv\left[  x_{1}:y_{1}\right]  $ and
$u_{2}\equiv\left[  x_{2}:y_{2}\right]  $, then a generalization of the
well-known \textbf{Fibonacci's identity }%
\[
\left(  x_{1}x_{2}+y_{1}y_{2}\right)  ^{2}+\left(  x_{1}y_{2}-x_{2}%
y_{1}\right)  ^{2}=\allowbreak\left(  x_{1}^{2}+y_{1}^{2}\right)  \left(
x_{2}^{2}+y_{2}^{2}\right)
\]
gives
\begin{align*}
q\left(  u_{1},u_{2}\right)   &  =1-\frac{\left(  ax_{1}x_{2}+b\left(
x_{1}y_{2}+x_{2}y_{1}\right)  +cy_{1}y_{2}\right)  ^{2}}{\left(  ax_{1}%
^{2}+2bx_{1}y_{1}+cy_{1}^{2}\right)  \left(  ax_{2}^{2}+2bx_{2}y_{2}%
+cy_{2}^{2}\right)  }\\
&  =\frac{\left(  ac-b^{2}\right)  \left(  x_{1}y_{2}-x_{2}y_{1}\right)  ^{2}%
}{\left(  ax_{1}^{2}+2bx_{1}y_{1}+cy_{1}^{2}\right)  \left(  ax_{2}%
^{2}+2bx_{2}y_{2}+cy_{2}^{2}\right)  }.
\end{align*}
This shows that $q\left(  u_{1},u_{2}\right)  =0$ precisely when $u_{1}%
=u_{2}.$ The pair $\left(  \mathbb{P}^{1},q\right)  $ denotes the projective
line together with the projective quadrance $q.$

The following is the fundamental formula for projective trigonometry in such a
one-dimensional setting, with a proof similar to the one in \cite{WildProj}.
In planar rational trigonometry this law is called the \textit{Triple spread
formula}, and is the analog of the fact that the sum of the angles in a
triangle in the real Cartesian plane is $\pi.$

\begin{theorem}
[Projective triple quad formula]If $u\equiv\left[  U\right]  ,$ $v\equiv
\left[  V\right]  $ and $w\equiv\left[  W\right]  $ are non-null p-points,
then the p-quadrances $q_{w}\equiv q\left(  u,v\right)  $, $q_{u}\equiv
q\left(  v,w\right)  $ and $q_{v}\equiv q\left(  u,w\right)  $ satisfy%
\[
\left(  q_{u}+q_{v}+q_{w}\right)  ^{2}=2\left(  q_{u}^{2}+q_{v}^{2}+q_{w}%
^{2}\right)  +4q_{u}q_{v}q_{w}.
\]

\end{theorem}

\begin{proof}
If we write $a_{U}\equiv U\cdot U$ and $b_{UV}\equiv U\cdot V$ then
\[
q_{u}=\frac{a_{V}a_{W}-b_{VW}^{2}}{a_{V}a_{W}}\qquad q_{v}=\frac{a_{U}%
a_{W}-b_{UW}^{2}}{a_{U}a_{W}}\qquad\mathrm{and}\qquad q_{w}=\frac{a_{U}%
a_{V}-b_{UV}^{2}}{a_{U}a_{V}}.
\]
The following is an algebraic identity in the abstract variables $a_{U}%
,a_{V},a_{W},$ $b_{UV},b_{UW}$ and $b_{VW}$:$\allowbreak$%
\begin{align*}
&  \left(  q_{u}+q_{v}+q_{w}\right)  ^{2}-2\left(  q_{u}^{2}+q_{v}^{2}%
+q_{w}^{2}\right)  -4q_{u}q_{v}q_{w}\\
&  =\allowbreak\left(  a_{U}a_{V}a_{W}+2b_{UV}b_{UW}b_{VW}-a_{U}b_{VW}%
^{2}-a_{V}b_{UW}^{2}-a_{W}b_{UV}^{2}\right) \\
&  \times\left(  2b_{UV}b_{UW}b_{VW}-a_{U}a_{V}a_{W}+a_{U}b_{VW}^{2}%
+a_{V}b_{UW}^{2}+a_{W}b_{UV}^{2}\right)  \allowbreak\allowbreak a_{W}%
^{-2}a_{V}^{-2}a_{U}^{-2}.
\end{align*}
But the first factor on the right hand side is the determinant%
\[%
\begin{vmatrix}
a_{U} & b_{UV} & b_{UW}\\
b_{UV} & a_{V} & b_{VW}\\
b_{UW} & b_{VW} & a_{W}%
\end{vmatrix}
=%
\begin{vmatrix}
U\cdot U & U\cdot V & U\cdot W\\
U\cdot V & V\cdot V & V\cdot W\\
U\cdot W & V\cdot W & W\cdot W
\end{vmatrix}
\]
which is zero since $U,V$ and $W$ are coplanar.$%
{\hspace{.1in} \rule{0.5em}{0.5em}}%
$
\end{proof}

\begin{example}
For an integer $k,$ the symmetric bilinear form
\[
\left(  x_{1},y_{1}\right)  \cdot\left(  x_{2},y_{2}\right)  \equiv x_{1}%
x_{2}+ky_{1}y_{2}%
\]
is non-degenerate provided that the characteristic of\textrm{\ }$\mathbb{F}$
does not divide $k,$ which we henceforth assume. In particular we assume that
$k$ is non-zero. If $u_{1}\equiv\left[  x_{1}:y_{1}\right]  $ and $u_{2}%
\equiv\left[  x_{2}:y_{2}\right]  $ then the associated p-quadrance $q_{k}$
is
\begin{align*}
q_{k}\left(  u_{1},u_{2}\right)   &  \equiv1-\frac{\left(  x_{1}x_{2}%
+ky_{1}y_{2}\right)  ^{2}}{\left(  x_{1}^{2}+ky_{1}^{2}\right)  \left(
x_{2}^{2}+ky_{2}^{2}\right)  }\\
&  =\frac{k\left(  x_{1}y_{2}-x_{2}y_{1}\right)  ^{2}}{\left(  x_{1}%
^{2}+ky_{1}^{2}\right)  \left(  x_{2}^{2}+ky_{2}^{2}\right)  }.%
\hspace{.1in}\diamond
\end{align*}

\end{example}

\section{Isometries of the projective line}

Let's now show how spread polynomials link to reflections and rotations in
one-dimensional geometry. Suppose $q$ is some fixed choice of p-quadrance on
$\mathbb{P}^{1}.$ An\textbf{\ isometry} of $\left(  \mathbb{P}^{1},q\right)  $
is a map $\tau:u\rightarrow u\tau$ that inputs and outputs non-null p-points,
and satisfies for any non-null p-points $u$ and $v$%
\[
q\left(  u,v\right)  =q\left(  u\tau,v\tau\right)  .
\]

For a $2\times2$ matrix
\[
T=%
\begin{pmatrix}
a & b\\
c & d
\end{pmatrix}
\]
representing the following linear transformation on $\mathbb{F}^{2}:$
\[
\left[  x,y\right]  \rightarrow\left[  x,y\right]
\begin{pmatrix}
a & b\\
c & d
\end{pmatrix}
=\left[  ax+cy,bx+dy\right]
\]
define the corresponding \textbf{projective transformation} $\tau$%
\[
\left[  x:y\right]  \tau=\left[  ax+cy:bx+dy\right]
\]
and denote it by%
\[
\tau\equiv%
\begin{bmatrix}
a & b\\
c & d
\end{bmatrix}
.
\]
The matrix for such a projective transformation is determined only up to a
scalar, so that%
\[%
\begin{bmatrix}
a & b\\
c & d
\end{bmatrix}
=%
\begin{bmatrix}
\lambda a & \lambda b\\
\lambda c & \lambda d
\end{bmatrix}
\]
for any non-zero number $\lambda$. We adapt the following from \cite{WildProj}.

\begin{theorem}
[Isometries of the projective line]An isometry of the projective space
$\left(  \mathbb{P}^{1},q_{k}\right)  $ must be a projective transformation of
one of the following types:%
\[%
\begin{tabular}
[c]{lllll}%
$\rho_{\left[  a:b\right]  }\equiv%
\begin{bmatrix}
a & b\\
-kb & a
\end{bmatrix}
$ &  & \textrm{or} &  & $\sigma_{\left[  a:b\right]  }\equiv%
\begin{bmatrix}
a & b\\
kb & -a
\end{bmatrix}
$%
\end{tabular}
\]
for some non-null projective point $\left[  a:b\right]  $.
\end{theorem}

\begin{proof}
Suppose that $\tau$ is an isometry of $\left(  \mathbb{P}^{1},q_{k}\right)  $
and that it sends $i_{1}\equiv\left[  1:0\right]  $ to $\left[  a:b\right]  $
and $i_{2}\equiv\left[  0:1\right]  $ to $\left[  c:d\right]  $. Since $i_{1}$
and $i_{2}$ are non-null p-points, $\left[  a:b\right]  $ and $\left[
c:d\right]  $ must also be non-null. Since $q_{k}\left(  i_{1},i_{2}\right)
=0,$ we must have
\[
ac+kbd=0.
\]
So $\left[  c:d\right]  =\left[  kb:-a\right]  .$ Now given an arbitrary
non-null projective point $u\equiv\left[  x:y\right]  $ with $u\tau
=v\equiv\left[  w:z\right]  $,%
\[
q_{k}\left(  u,i_{1}\right)  =\frac{ky^{2}}{x^{2}+ky^{2}}=q_{k}\left(
v,\left[  a:b\right]  \right)  =\frac{k\left(  az-bw\right)  ^{2}}{\left(
a^{2}+kb^{2}\right)  \left(  w^{2}+kz^{2}\right)  }%
\]
and%
\[
q_{k}\left(  u,i_{2}\right)  =\frac{x^{2}}{x^{2}+ky^{2}}=q_{k}\left(
v,\left[  kb:-a\right]  \right)  =\frac{\left(  aw+kbz\right)  ^{2}}{\left(
a^{2}+kb^{2}\right)  \left(  w^{2}+kz^{2}\right)  }.
\]
Comparing these two equations gives%
\[
x^{2}:y^{2}=\left(  aw+kbz\right)  ^{2}:\left(  az-bw\right)  ^{2}%
\]
so that either
\[
x:y=\left(  aw+kbz\right)  :\left(  az-bw\right)  \qquad\mathrm{or}\qquad
x:y=-\left(  aw+kbz\right)  :\left(  az-bw\right)  .
\]
Using matrix notation, either%
\[
\left[  x:y\right]  =\left[  w:z\right]
\begin{bmatrix}
a & -b\\
kb & a
\end{bmatrix}
\qquad\mathrm{or}\qquad\left[  x:y\right]  =\left[  w:z\right]
\begin{bmatrix}
-a & -b\\
-kb & a
\end{bmatrix}
.
\]
Inverting gives either%
\[
\tau=%
\begin{bmatrix}
a & b\\
-kb & a
\end{bmatrix}
\qquad\mathrm{or}\qquad\tau=%
\begin{bmatrix}
a & b\\
kb & -a
\end{bmatrix}
.
\]
If $\left[  a:b\right]  $ is non-null, then both of these are isometries, due
to the equations
\[%
\begin{bmatrix}
a & b\\
-kb & a
\end{bmatrix}%
\begin{bmatrix}
1 & 0\\
0 & k
\end{bmatrix}%
\begin{bmatrix}
a & b\\
-kb & a
\end{bmatrix}
^{T}=\allowbreak%
\begin{bmatrix}
a^{2}+b^{2}k & 0\\
0 & a^{2}k+b^{2}k^{2}%
\end{bmatrix}
=%
\begin{bmatrix}
1 & 0\\
0 & k
\end{bmatrix}
\]
and%
\[%
\begin{bmatrix}
a & b\\
kb & -a
\end{bmatrix}%
\begin{bmatrix}
1 & 0\\
0 & k
\end{bmatrix}%
\begin{bmatrix}
a & b\\
kb & -a
\end{bmatrix}
^{T}=\allowbreak%
\begin{bmatrix}
a^{2}+b^{2}k & 0\\
0 & a^{2}k+b^{2}k^{2}%
\end{bmatrix}
=%
\begin{bmatrix}
1 & 0\\
0 & k
\end{bmatrix}
.%
{\hspace{.1in} \rule{0.5em}{0.5em}}%
\]

\end{proof}

We call
\[%
\begin{tabular}
[c]{lllll}%
$\rho_{\left[  a:b\right]  }\equiv%
\begin{bmatrix}
a & b\\
-kb & a
\end{bmatrix}
$ &  & \textrm{and} &  & $\sigma_{\left[  a:b\right]  }\equiv%
\begin{bmatrix}
a & b\\
kb & -a
\end{bmatrix}
$%
\end{tabular}
\]
respectively a \textbf{projective rotation} and a \textbf{projective
reflection}. These realizations as projective matrices extend isometries also
to null points.

The \textbf{identity transformation}
\[
\rho_{\left[  1:0\right]  }\equiv%
\begin{bmatrix}
1 & 0\\
0 & 1
\end{bmatrix}
\]
is a projective rotation.

Our convention for compositions is $u\left(  \tau_{1}\tau_{2}\right)
\equiv\left(  u\tau_{1}\right)  \tau_{2}.$

\begin{theorem}
[Composition of isometries]For any non-null p-points $\left[  a:b\right]  $
and $\left[  c:d\right]  $%
\[%
\begin{tabular}
[c]{lll}%
$\sigma_{\left[  a:b\right]  }\sigma_{\left[  c:d\right]  }=\rho_{\left[
ac+kbd:ad-bc\right]  }$ &  & $\rho_{\left[  a:b\right]  }\rho_{\left[
c:d\right]  }=\rho_{\left[  ac-kbd:ad+bc\right]  }$\\
$\rho_{\left[  a:b\right]  }\sigma_{\left[  c:d\right]  }=\sigma_{\left[
ac+kbd:ad-bc\right]  }$ &  & $\sigma_{\left[  a:b\right]  }\rho_{\left[
c:d\right]  }=\sigma_{\left[  ac-kbd:ad+bc\right]  .}$%
\end{tabular}
\]

\end{theorem}

\begin{proof}
This is a straightforward verification. Another generalization of the
Fibonacci identities,%
\[
\left(  ac+kbd\right)  ^{2}+k\left(  ad-bc\right)  ^{2}=\left(  a^{2}%
+kb^{2}\right)  \left(  c^{2}+kd^{2}\right)  =\left(  ac-kbd\right)
^{2}+k\left(  ad+bc\right)  ^{2}%
\]
$\allowbreak$shows that the resultant isometries are also associated to
non-null points.$%
{\hspace{.1in} \rule{0.5em}{0.5em}}%
$
\end{proof}

Define $G\equiv G\left(  q_{k}\right)  $ to be the group of isometries of
$q_{k}$, and distinguish the subgroup $G_{e}\equiv G_{e}\left(  q_{k}\right)
$ of projective rotations $\rho_{\left[  a:b\right]  }.$ These latter are
naturally in bijection with the non-null p-points. From the Composition of
isometries theorem, the subgroup $G_{e}$ is commutative.

The coset $G_{o}\equiv G_{o}\left(  q_{k}\right)  $ consists of projective
reflections $\sigma_{\left[  a:b\right]  },$ and these too are also naturally
in bijection with the non-null p-points. The group $G$ naturally acts on the
space of non-null p-points. This action is transitive since $\sigma_{\left[
a:b\right]  }$ and $\rho_{\left[  a:b\right]  }$ both send $\left[
1:0\right]  $ to $\left[  a:b\right]  $.

Since projective rotations are in bijection with non-null p-points, we can
\textit{transfer the multiplicative structure }%
\[
\rho_{\left[  a:b\right]  }\rho_{\left[  c:d\right]  }=\rho_{\left[
ac-kbd:ad+bc\right]  }%
\]
\textit{of projective rotations to non-null p-points}. So for $u\equiv\left[
a:b\right]  $ and $v\equiv\left[  c:d\right]  $ non-null p-points, define
their $k$\textbf{-product} by the rule:%

\[
uv=\left[  a:b\right]  \left[  c:d\right]  \equiv\left[  ac-kbd:ad+bc\right]
.
\]
When $k=1$ this multiplication is familiar from the two-dimensional setting of
complex numbers. Note however that we are here working in the one-dimensional
situation, over a general field, and allowing different values of $k.$ The
resulting group is commutative, has \textbf{identity}
\[
e\equiv\left[  1:0\right]
\]
and the\textbf{\ inverse} of $u\equiv\left[  a:b\right]  $ is $u^{-1}%
\equiv\left[  a:-b\right]  $. For a non-null projective point $u,$ we let
$u^{n}\equiv uu\cdots u$ ($n$ times) denote the $n$\textbf{-th} \textbf{power}
of $u.$ Of course this depends on the prior choice $q=q_{k}$ of p-quadrance.

\section{Spreads of p-points}

Working in $\left(  \mathbb{P}^{1},q_{k}\right)  $, define the $k$%
\textbf{-spread} $s_{k}\left(  u\right)  $ of the non-null p-point
$u\equiv\left[  a:b\right]  $ to be the number
\[
s_{k}\left(  u\right)  \equiv q_{k}\left(  e,u\right)  =\allowbreak
\frac{kb^{2}}{a^{2}+kb^{2}}.
\]
Then $s_{k}\left(  u\right)  =0$ precisely when $u=\left[  1:0\right]  =e$,
and $s_{k}\left(  u\right)  =1$ precisely when $u=\left[  0:1\right]  $. Since
multiplication by $v$ is the same as applying the projective rotation
$\rho_{v},$ multiplication is an isometry, so that for any non-null p-points
$v$ and $u,$
\[
q_{k}\left(  v,vu\right)  =s_{k}\left(  u\right)  =\frac{kb^{2}}{a^{2}+kb^{2}%
}.
\]

The next result connects spread polynomials and powers.

\begin{theorem}
[Spread of a power]In $\left(  \mathbb{P}^{1},q_{k}\right)  $, if $u$ is a
non-null projective point with $s_{k}\left(  u\right)  \equiv s$, then for any
natural number $n,$
\[
s_{k}\left(  u^{n}\right)  =S_{n}\left(  s\right)  .
\]

\end{theorem}

\begin{proof}
Suppose that $u\equiv\left[  a:b\right]  $ so that
\[
s_{k}\left(  u\right)  =\frac{kb^{2}}{a^{2}+b^{2}k}\equiv s.
\]
We know from the Isometries of the projective line theorem that the projective
rotation $\rho_{u}$ has the form%
\[
\rho_{u}=%
\begin{bmatrix}
a & b\\
-kb & a
\end{bmatrix}
.
\]
We diagonalize the matrix $%
\begin{pmatrix}
a & b\\
-kb & a
\end{pmatrix}
$ using a number $r$ satisfying $r^{2}=-k.$ If the field $\mathbb{F}$ does not
contain such a number, the following equations take place in the quadratic
extension $\mathbb{F}\left(  r\right)  $. First verify that
\[%
\begin{pmatrix}
a & b\\
-kb & a
\end{pmatrix}
=%
\begin{pmatrix}
1 & -1\\
r & r
\end{pmatrix}%
\begin{pmatrix}
a+br & 0\\
0 & a-br
\end{pmatrix}%
\begin{pmatrix}
1 & -1\\
r & r
\end{pmatrix}
^{-1}.
\]
Then for any natural number $n,\allowbreak\allowbreak$
\begin{align*}%
\begin{pmatrix}
a & b\\
-kb & a
\end{pmatrix}
^{n}  &  =%
\begin{pmatrix}
1 & -1\\
r & r
\end{pmatrix}%
\begin{pmatrix}
\left(  a+br\right)  ^{n} & 0\\
0 & \left(  a-br\right)  ^{n}%
\end{pmatrix}%
\begin{pmatrix}
1 & -1\\
r & r
\end{pmatrix}
^{-1}\\
&  =%
\begin{pmatrix}
\frac{1}{2}\left(  a+br\right)  ^{n}+\frac{1}{2}\left(  a-br\right)  ^{n} &
\frac{1}{2r}\left(  a+br\right)  ^{n}-\frac{1}{2r}\left(  a-br\right)  ^{n}\\
-\frac{r}{2}\left(  a+br\right)  ^{n}-\frac{r}{2}\left(  a-br\right)  ^{n} &
\frac{1}{2}\left(  a+br\right)  ^{n}+\frac{1}{2}\left(  a-br\right)  ^{n}%
\end{pmatrix}
.
\end{align*}
Taking the proportion determined by the first row gives
\[
u^{n}=\left[  r\left(  a+br\right)  ^{n}+r\left(  a-br\right)  ^{n}:\left(
a+br\right)  ^{n}-\left(  a-br\right)  ^{n}\right]
\]
so
\[
t_{n}\equiv s_{k}\left(  u^{n}\right)  =-\frac{\left(  \left(  a+br\right)
^{n}-\left(  a-br\right)  ^{n}\right)  ^{2}}{4\left(  a+br\right)  ^{n}\left(
a-br\right)  ^{n}}=-\frac{\left(  A^{n}-B^{n}\right)  ^{2}}{4A^{n}B^{n}}%
\]
where $A\equiv a+br$ and $B\equiv a-br$.

Clearly $t_{0}=0,$ and since $r^{2}=-k,$
\[
t_{1}=\frac{kb^{2}}{a^{2}+b^{2}k}=s.
\]
To show that $t_{n}=S_{n}\left(  s\right)  $ for all natural numbers $n,$ we
establish the identity%
\begin{equation}
t_{n}-2\left(  1-2s\right)  t_{n-1}+t_{n-2}-2s=0 \label{Spread Identity}%
\end{equation}
for all $n\geq2.$ The left hand side of (\ref{Spread Identity}) is
\[
-\frac{\left(  A^{n}-B^{n}\right)  ^{2}}{4A^{n}B^{n}}+2\left(  1+\frac
{2\left(  A-B\right)  ^{2}}{4AB}\right)  \frac{\left(  A^{n-1}-B^{n-1}\right)
^{2}}{4A^{n-1}B^{n-1}}-\frac{\left(  A^{n-2}-B^{n-2}\right)  ^{2}}%
{4A^{n-2}B^{n-2}}+\frac{2\left(  A-B\right)  ^{2}}{4AB}%
\]
and if we remove a factor of $\left(  4A^{n}B^{n}\right)  ^{-1}$ this becomes
\begin{align*}
&  -A^{2n}+2A^{n}B^{n}-B^{2n}+\left(  A^{2}+B^{2}\right)  \left(  A^{2\left(
n-1\right)  }-2A^{n-1}B^{n-1}+B^{2\left(  n-1\right)  }\right) \\
&  -A^{2}B^{2}\left(  A^{2\left(  n-2\right)  }-2A^{n-2}B^{n-2}+B^{2\left(
n-2\right)  }\right)  +2A^{n-1}B^{n-1}\left(  A^{2}-2AB+B^{2}\right)  ,
\end{align*}
which after expansion is identically zero, so holds independent of the
particular choices of $A$ and $B.%
{\hspace{.1in} \rule{0.5em}{0.5em}}%
$
\end{proof}

\begin{theorem}
[Spread composition]For any natural numbers $n$ and $m,$
\[
S_{n}\circ S_{m}=S_{nm}.
\]

\end{theorem}

\begin{proof}
Working over the rational numbers with the Euclidean quadrance $q,$ the Spread
of a power theorem shows that if $s\equiv s\left(  u\right)  $ for some
p-point $u$, then $s\left(  u^{m}\right)  =S_{m}\left(  s\right)  .$ It
follows that
\[
S_{n}\left(  S_{m}\left(  s\right)  \right)  =s\left(  \left(  u^{m}\right)
^{n}\right)  =s\left(  u^{mn}\right)  =S_{nm}\left(  s\right)  .
\]
Since this holds for more than $nm$ different values of $s,$ and both
$S_{n}\circ S_{m}$ and $S_{nm}$ are polynomials of degree $nm,$ we conclude
that
\[
S_{n}\circ S_{m}=S_{nm}.%
{\hspace{.1in} \rule{0.5em}{0.5em}}%
\]

\end{proof}

\section{Spread numbers in a prime field}

Fix an integer $k$ which is not zero in the field $\mathbb{F}_{p}$ of
characteristic $p\neq2$, in other words which is not divisible by $p.$ Define
a number $a$ in $\mathbb{F}_{p}$ to be a $k$\textbf{-spread number} precisely
when $a\left(  1-a\right)  $ is $k$ times a square. If $k=1$ this agrees with
our earlier usage.

For $p$ an odd prime, the finite prime field $\mathbb{F}_{p}$ contains an
equal number of non-zero squares and non-squares. This follows for example
from the standard fact that the multiplicative group of a finite field is
cyclic. Note that for $p\equiv3$ $\operatorname{mod}4$ every number is either
a square or the negative of a square, since $-1$ is not a square, and so every
number is either a $1$-spread number or a $\left(  -1\right)  $-spread number.
However for $p\equiv1$ $\operatorname{mod}4$ the $1$-square numbers and the
$\left(  -1\right)  $-square numbers agree.

Every number $a$ in a field is a $k$-spread number for at least one $k,$ for
example $0$ and $1$ are $k$-spread numbers for all $k$, and for other $a$ we
may choose $k\equiv a\left(  1-a\right)  $. The next result generalizes the
Spread number theorem in \cite[Theorem 34]{Wild}.

\begin{theorem}
[Spread number]For any non-null p-points $u$ and $v$ in $\left(
\mathbb{P}^{1},q_{k}\right)  $, the p-quadrance $q_{k}\left(  u,v\right)  $ is
a $k$-spread number, and conversely for every $k$-spread number $q$ there
exist non-null p-points $u$ and $v$ with $q_{k}\left(  u,v\right)  =q$, and so
there exists a non-null p-point $w$ with $s_{k}\left(  w\right)  =q.$
\end{theorem}

\begin{proof}
If $u\equiv\left[  x_{1}:y_{1}\right]  $ and $v\equiv\left[  x_{2}%
:y_{2}\right]  $ then
\[
q_{k}\left(  u,v\right)  =\frac{k\left(  x_{1}y_{2}-x_{2}y_{1}\right)  ^{2}%
}{\left(  x_{1}^{2}+ky_{1}^{2}\right)  \left(  x_{2}^{2}+ky_{2}^{2}\right)
}\equiv q
\]
in which case%
\[
q\left(  1-q\right)  =k\frac{\left(  x_{1}x_{2}+ky_{1}y_{2}\right)
^{2}\allowbreak\left(  x_{2}y_{1}-x_{1}y_{2}\right)  ^{2}}{\left(  x_{2}%
^{2}+ky_{2}^{2}\right)  ^{2}\left(  x_{1}^{2}+ky_{1}^{2}\right)  ^{2}}%
\]
which is $k$ times a square, so $q$ is a $k$-spread number. Conversely suppose
that $q$ is a $k$-spread number, so that $q\left(  1-q\right)  =kr^{2}$ for
some number $r$ in the field. If $q=1$ then it is the p-quadrance between
$\left[  1:0\right]  $ and $\left[  0:1\right]  $. Otherwise the p-quadrance
between the p-points $u\equiv\left[  1:0\right]  $ and $v\equiv\left[
1-q:r\right]  $ is%
\[
k\frac{r^{2}}{\left(  1-q\right)  ^{2}+kr^{2}}=\frac{q\left(  1-q\right)
}{\left(  1-q\right)  ^{2}+q\left(  1-q\right)  }=q.
\]
Note that
\[
\left(  1-q\right)  ^{2}+kr^{2}=\left(  1-q\right)  ^{2}+q\left(  1-q\right)
=\allowbreak1-q
\]
is indeed non-zero so both $u$ and $v$ are non-null.

Once we have $u$ and $v$ we can multiply both by $u^{-1}$ to obtain $1$ and
$w\equiv u^{-1}v$ so that
\[
q_{k}\left(  1,w\right)  =s_{k}\left(  w\right)  =q.%
{\hspace{.1in} \rule{0.5em}{0.5em}}%
\]

\end{proof}

To determine $k$-spread numbers in a prime field $\mathbb{F}_{p}$ it suffices
to know the squares in the field. In terms of the Legendre symbol%
\[%
\begin{pmatrix}
a\\
p
\end{pmatrix}
=\left\{
\begin{array}
[c]{cc}%
1 & \qquad\qquad\text{if }a\text{ is a square in }\mathbb{F}_{p}\\
-1 & \text{otherwise}%
\end{array}
\right.
\]
$q$ is a $k$-spread number precisely when
\begin{equation}%
\begin{pmatrix}
kq\left(  1-q\right) \\
p
\end{pmatrix}
=%
\begin{pmatrix}
-k\\
p
\end{pmatrix}%
\begin{pmatrix}
q\\
p
\end{pmatrix}%
\begin{pmatrix}
q-1\\
p
\end{pmatrix}
=1. \label{Legendre}%
\end{equation}
So it suffices to know whether or not $q$, $q-1$ and $-k$ are squares. If
$k=1$ (Euclidean geometry) then the above gives a straightforward recipe for
finding spread numbers from square numbers.

\begin{example}
In $\mathbb{F}_{13},$ $%
\begin{pmatrix}
-1\\
p
\end{pmatrix}
=1$ and since the square numbers are $0,1,3,4,9,10$ and $12,$ the spread
numbers are $q=0,1,4,6,7,8$ and $10,$ as these are those $q$ (aside from $0$
and $1$) whose Legendre symbol agrees with that of $q-1.$ Whereas in
$\mathbb{F}_{11},$ $%
\begin{pmatrix}
-1\\
p
\end{pmatrix}
=-1$ and since the squares are $0,1,3,4,5$ and $9$, the spread numbers are
$0,1,2,3,6,9$ and $10,$ those $q$ (aside from $0$ and $1$) whose Legendre
symbol disagree with that of $q-1.%
\hspace{.1in}\diamond
$
\end{example}

\section{Spread periodicity}

Suppose $\mathbb{F=F}_{p}$ for some odd prime $p.$ In $\left(  \mathbb{P}%
^{1},q_{k}\right)  $ there are exactly $p+1$ p-points, but the number of null
p-points depends also on $k;$ there are null p-points precisely when
\[
x^{2}+ky^{2}=0
\]
has non-zero solutions. This occurs precisely when $-k$ is a square modulo
$p,$ and in this case there are exactly $2$ null p-points, and so the order of
the group of rotations $G_{e}$ is $p-1$. Otherwise there are no null p-points
and $G_{e}$ has order $p+1.$

We regard an element of $\mathbb{F}_{p}$ as an expression of the form $a/b$,
where $a$ and $b$ are integers with $b$ not divisible by $p,$ with the
convention that $a/b=c/d$ precisely when
\[
ad-bc\text{ is divisible by }p.
\]
Thus every rational number $a/b$ with $b$ not divisible by $p$ represents also
an element of $\mathbb{F}_{p}.$ So the values of the spread polynomial
$S_{n}\left(  s\right)  $ over $\mathbb{F}_{p}$ are just obtained by reducing
$\operatorname{mod}p.$

\begin{theorem}
[Spread periodicity]Fix an odd prime $p.$ For any element $s$ in
$\mathbb{F}_{p}$, there exists a positive integer $m$ with the property that
$S_{n}\left(  s\right)  =0$ in $\mathbb{F}_{p}$ precisely when $n$ is a
multiple of $m.$ Furthermore $m$ divides either $p-1$ or $p+1.$
\end{theorem}

\begin{proof}
Choose a non-zero integer $k$ such that $s$ is a $k$-spread number in
$\mathbb{F}_{p}.$ The Spread number theorem then asserts that there is a
non-null projective point $u$ in the projective line $\left(  \mathbb{P}%
^{1},q_{k}\right)  $ over $\mathbb{F}_{p}$ with $s_{k}\left(  u\right)  =s.$
Multiplication by $u$ is a projective rotation $\rho_{u}$, and by the Spread
of a power theorem $s_{k}\left(  u^{n}\right)  =S_{n}\left(  s\right)  $. So
$S_{n}\left(  s\right)  =0$ precisely when $u^{n}=e=\left[  1,0\right]  $.
Since $\rho_{u}$ belongs to the finite commutative group $G_{e},$
$S_{n}\left(  s\right)  =0$ precisely when $n$ is a multiple of the order $m$
of $\rho_{u}$ in $G_{e}.$ This $m$ divides the order of the group $G_{e},$
which is either $p-1$ or $p+1.%
{\hspace{.1in} \rule{0.5em}{0.5em}}%
$
\end{proof}

\begin{corollary}
\textit{For any rational number }$s=a/b$\textit{, any prime }$p$\textit{\ not
dividing }$b$ occurs \textit{infinitely often} as a factor of the numbers
$S_{n}\left(  s\right)  $ for $n=1,2,3,\cdots.$
\end{corollary}

\begin{proof}
This is an immediate consequence of the theorem.$%
{\hspace{.1in} \rule{0.5em}{0.5em}}%
$
\end{proof}

To illustrate the theorem for $s=1/3$, we will discuss the numbers $m\left(
5\right)  $, $m\left(  7\right)  $ and $m\left(  19\right)  $.

\begin{example}
In $\mathbb{F}_{5}$ the squares are $0,1$ and $4$, while the $1$-spread
numbers are $0,1$ and $3.$ In this field $s=1/3=2$ and so $s\left(
1-s\right)  =3,$ which is not a square, but $s\left(  1-s\right)  =k\times
r^{2} $ for $k=2$ and $r=2.$ That means that $s$ is a $2$-spread number (and
hence also a $3$-spread number). There are no null p-points for either of the
projective quadrances $q_{2}$ or $q_{3},$ since $-2$ and $-3$ are not squares
in $\mathbb{F}_{5}.$ Hence $G_{e}$ has $6$ elements in either case, and so the
order $m\left(  5\right)  $ divides $6.$ In fact $m\left(  5\right)  =3.$ A
projective point with $2$-spread equal to $2$ is $u\equiv\left[  1-s:r\right]
=\left[  2:1\right]  $. Then $u^{2}=\left[  1:2\right]  $ and $u^{3}=\left[
1:0\right]  $.$%
\hspace{.1in}\diamond
$
\end{example}

\begin{example}
In $\mathbb{F}_{7}$ the squares are $0,1,2$ and $4$, while the $1$-spread
numbers are $0,1,3,4$ and $5.$ In this field $s=1/3=5$ is a spread number
since $s\left(  1-s\right)  =k\times r^{2}$ with $k=1$ and $r=1.$ Then
$\left(  \mathbb{P}^{1},q_{1}\right)  $ has no null p-points since $-1$ is not
a square, so $G_{e}$ has order $8.$ Thus $m\left(  7\right)  $ divides $8.$ In
fact $m\left(  7\right)  =8.$ A projective point with $1$-spread equal to $5$
is $u\equiv\left[  1-s:r\right]  =\left[  3:1\right]  $. Then the powers of
$u$ are $u^{2}=\left[  6:1\right]  $, $u^{3}=\left[  5:1\right]  $,
$u^{4}=\left[  0:1\right]  $, $u^{5}=\left[  2:1\right]  $, $u^{6}=\left[
1:1\right]  $, $u^{7}=\left[  4:1\right]  $ and $u^{8}=\left[  1:0\right]  $.$%
\hspace{.1in}\diamond
$
\end{example}

\begin{example}
In $\mathbb{F}_{19}$ the squares are%
\[
0,1,4,5,6,7,9,11,16,17
\]
while the $1$-spread numbers are
\[
0,1,2,4,8,9,10,11,12,16,18.
\]
In this field $s=13$ is not a $1$-spread number but $s\left(  1-s\right)
=\allowbreak\allowbreak15=k\times r^{2}$ with $k=-1$ and $r=2,$ so $s$ is a
$\left(  -1\right)  $-spread number. Then $\left(  \mathbb{P}^{1}%
,q_{-1}\right)  $ has null p-points since $1$ is a square, so $G_{e}$ has
order $18.$ Thus $m\left(  19\right)  $ divides $18.$ In fact $m\left(
19\right)  =9.$ A projective point with $\left(  -1\right)  $-spread equal to
$13$ is $u\equiv\left[  1-s:r\right]  =\left[  13:1\right]  $. Then the powers
of $u$ are $u^{2}=\left[  8:1\right]  $, $u^{3}=\left[  5:1\right]  $,
$u^{4}=\left[  10:1\right]  $, $u^{5}=\left[  9:1\right]  $, $u^{6}=\left[
14:1\right]  $, $u^{7}=\left[  11:1\right]  $, $u^{8}=\left[  6:1\right]  $
and $u^{9}=\left[  1:0\right]  $.$%
\hspace{.1in}\diamond
$
\end{example}

\subsection{The special spreads}

The spreads%
\[
s=0,\frac{1}{4},\frac{1}{2},\frac{3}{4},1
\]
play a special role in rational trigonometry. Over the real numbers, they
correspond to angles of $0,30^{\circ},45^{\circ},60^{\circ}$ and $90^{\circ}$
respectively, and also to the supplements of these angles. Over a fixed field,
any spread polynomial evaluated at any one of these five values of $s$ will
also yield one of these five values.

For example if $s=1/4$ then $S_{2}\left(  s\right)  =3/4,$ $S_{3}\left(
s\right)  =1,$ $S_{4}\left(  s\right)  =3/4,$ $S_{5}\left(  s\right)  =1/4$
and $S_{6}\left(  s\right)  =0.$ It is not hard to check that
\[
S_{6}\left(  s\right)  =0
\]
for any one of these values of $s.$ Since this is true over the rationals, it
is true over any other field not of characteristic two.

\begin{example}
In $\mathbb{F}_{7},$ $s=1/4=2$ which is a $3$-spread number since $s\left(
1-s\right)  =5=k\times r^{2}$ with $k=3$ and $r=2.$ Then $q_{3}$ has null
p-points since $-3$ is a square, so $G_{e}$ has order $6.$ Thus $m\left(
7\right)  $ divides $6.$ In fact $m\left(  7\right)  =6.$ A projective point
with $3$-spread equal to $2$ is $u=\left[  1-s:r\right]  =\left[  3,1\right]
$. Then the powers of $u$ are $u^{2}=\left[  1:1\right]  $, $u^{3}=\left[
0:1\right]  $, $u^{4}=\left[  6:1\right]  $, $u^{5}=\left[  4:1\right]  $ and
$u^{6}=\left[  1:0\right]  $.$%
\hspace{.1in}\diamond
$
\end{example}

\begin{example}
In $\mathbb{F}_{11},$ $s=1/4=3$ which is a $1$-spread number and $s\left(
1-s\right)  =5=4^{2}$. Then $\left(  \mathbb{P}^{1},q_{1}\right)  $ has no
null p-points since $-1$ is not a square, so $G_{e}$ has order $12.$ Thus
$m\left(  11\right)  $ divides $12.$ In fact $m\left(  11\right)  =6.$ A
projective point with $1$-spread equal to $3$ is $u=\left[  1-s:r\right]
=\left[  -2:4\right]  =\left[  5:1\right]  $. Then the powers of $u$ are
$u^{2}=\left[  9:1\right]  $, $u^{3}=\left[  0:1\right]  $, $u^{4}=\left[
2:1\right]  $, $u^{5}=\left[  6:1\right]  $ and $u^{6}=\left[  1:0\right]  $.$%
\hspace{.1in}\diamond
$
\end{example}

\section{Conclusion}

The results of this paper indicate that there are further rich number
theoretical aspects of the spread polynomials. Some of these will be studied
in future work, particularly the parallels between Spread polynomials and
Chebyshev polynomials, and the close affinity with \textit{cyclotomic
polynomials}.

\end{document}